\documentclass[11pt,a4paper, final, twoside]{article}
\usepackage{amsmath}
\usepackage{fancyhdr}
\usepackage{amsthm}
\usepackage{amsfonts}
\usepackage{amssymb}
\usepackage{amscd}
\usepackage{amsthm}
\usepackage{graphicx}
\usepackage{afterpage}
\usepackage[colorlinks=true, urlcolor=blue,  linkcolor=blue, citecolor=blue]{hyperref}

\newtheorem{theorem}{Theorem}

\newtheorem{definition}[theorem]{Definition}

\newtheorem{proposition}[theorem]{Proposition}
\newtheorem{remark}[theorem]{Remark}

\numberwithin{equation}{section}

\begin{document}
\hyphenpenalty=10000

\begin{center}
{\Large \textbf{Local Spectrum of a Family of Operators }}\\[5mm]
{\large {Simona Macovei*  }\\[10mm]
}
\end{center}

{\footnotesize \textbf{Abstract}. Starting from the classic definitions of resolvent set and spectrum of a linear bounded operator on a Banach space, we introduce the local resolvent set and local spectrum, the local spectral space and the single-valued extension property of a family of linear bounded operators on a Banach space. Keeping the analogy with the classic case, we extend some of the known results from the case of a linear bounded operator to the case of a family of linear bounded operators on a Banach space.}
\footnote{\textsf{*simonamacovei@yahoo.com}}

\afterpage{
\fancyhead{} \fancyfoot{} 
\fancyhead[LE, RO]{\bf\thepage}
\fancyhead[LO]{\small Local Spectrum of a Family of Operators}
\fancyhead[RE]{\small Simona Macovei  }
}
 
\section{Introduction}

\noindent 

\noindent Let \textit{X} be a complex Banach space and $B(X)$ the Banach algebra of linear bounded operators on \textit{X}. Let \textit{T} be a linear bounded operator on \textit{X}. The\textit{ norm} of \textit{T} is

\noindent 
\[\left\|T\right\|=sup\left\{\left\|Tx\right\||\ x\in X,\ \left\|x\right\|\le 1\right\}.\]

\noindent The\textit{ spectrum} of an operator $T\in B(X)$ is defined as the set 

\noindent 
\[Sp\left(T\right){\rm =}{\mathbb C}\backslash r(T),\]

\noindent where $r(T)$ is the \textit{resolvent set} of \textit{T} and consists in all complex numbers $\lambda \in {\mathbb C}$ for which the operator $\lambda I-T$ is bijective on \textit{X}. 

\noindent 

\noindent An operator $T\in B(X)$ is said to have the \textit{single-valued extension property} if for any analytic function $f:D_f\to X$, where $D_f\subset {\mathbb C}$ is open, with $\left(\lambda I-T\right)f\left(\lambda \right)\equiv 0$, it results $f\left(\lambda \right)\equiv 0$.

\noindent For an operator $T\in B(X)$ having the single-valued extension property and for $x\in X$ we can consider the set $r_T\left(x\right)$ of elements ${\lambda }_0\in {\mathbb C}$ such that there is an analytic function $\lambda \mapsto x(\lambda )$ defined in a neighborhood of ${\lambda }_0$ with values in \textit{X}, which verifies $\left(\lambda I-T\right)x\left(\lambda \right)\equiv x$. The set $r_T\left(x\right)$ is said \textit{the local resolvent set of} \textit{T at  }$x$, and the set ${Sp}_T\left(x\right){\rm =}{\mathbb C} \backslash r_T\left(x\right)$ is called \textit{the local spectrum of T at }$x$\textit{.}

\noindent An analytic function $f_x:D_x\to X$, where $D_x\subset {\mathbb C}$ is open, is said the \textit{analytic  extension }of function $\lambda \mapsto R\left(\lambda ,T\right)x$ if $r(T)\subset D_x$ and $\left(\lambda I-T\right)f_x\left(\lambda \right)\equiv x$.

\noindent If \textit{T} has the single-valued extension property, then, for any $x\in X$ there is a unique \textit{maximal analytic extension} of function $\lambda \mapsto R\left(\lambda ,T\right)x:r_T\left(x\right)\to X$, referred from now as $x\left(\lambda \right)$. Moreover, $r_T\left(x\right)$ is an open set of $C$ and $r(T)\subset r_T\left(x\right)$.

\noindent Let 
\[X_T\left(a\right)=\left\{{x\in X|Sp}_T\left(x\right)\subset a\right\}\]

\noindent be the \textit{local spectral space} of \textit{T }for all sets $a\subset {\mathbb C}$. The space $X_T\left(a\right)$ is a linear subspace (not necessary closed) of \textit{X}.

\noindent 

\noindent Two operators $T,S\in B(X)$ are \textit{quasinilpotent equivalent} if

\noindent 
\[{\mathop{\lim }_{n\to \infty } {\left\|{(T-S)}^{\left[n\right]}\right\|}^{\frac{1}{n}}\ }={\mathop{\lim }_{n\to \infty } {\left\|{(S-T)}^{\left[n\right]}\right\|}^{\frac{1}{n}}\ }=0,\]

\noindent where ${\left(T-S\right)}^{\left[n\right]}=\sum^n_{k=0}{{\left(-1\right)}^{n-k}C^n_kT^kS^{n-k}}$, for any $n\in {\mathbb N}$.

\noindent The quasinilpotent equivalence relation is an equivalence relation (i.e. is reflexive, symmetric and transitive) on $B(X)$.

\noindent 
\begin {theorem}
\noindent Let $T,S\in B(X)$ be two quasinilpotent equivalent operators. Then

\noindent i) $Sp\left(T\right)=Sp\left(S\right)$$;$

\noindent ii)\textit{T} has the single-valued extension property if an only if \textit{S} has the single-valued extension property. Moreover, ${Sp}_T\left(x\right)={Sp}_S\left(x\right)$.
\end {theorem}
\noindent 

\noindent For an easier understanding of the results from this paper, we recall some definitions and results introduced by author in ''\textit{Spectrum of a Family of Operators}'' $\left[6\right]$.

\noindent We say that two families of operators $\left\{S_h\right\},\ \left\{T_h\right\}\ \subset B(X)$, with $h\in \left.(0,1\right],$ are  \textit{asymptotically equivalent} if
\[{\mathop{lim}_{h\to 0} \left\|S_h-T_h\right\|=0\ }.\]

\noindent Two families of operators $\left\{ S_{h}\right\} ,\ \left\{
T_{h}\right\} \ \subset B(X)$, with $h\in \left. (0,1\right] $, are
\textit{asymptotically quasinilpotent (spectral) equivalent} if\noindent 
\begin{equation*}
{\mathop{\lim}_{n\rightarrow \infty }{{\mathop{\lim \sup}_{h\rightarrow
0}\left\Vert {\left( S_{h}-T_{h}\right) }^{\left[ n\right] }\right\Vert \ }}%
^{\frac{1}{n}}\ }={\mathop{\lim}_{n\rightarrow \infty }{{\mathop{\lim \sup}%
_{h\rightarrow 0}\left\Vert {\left( T_{h}-S_{h}\right) }^{\left[ n\right]
}\right\Vert \ }}^{\frac{1}{n}}\ }=0.
\end{equation*}

\noindent The asymptotic (quasinilpotent) equivalence between two families of operators $\left\{S_h\right\}$, $\left\{T_h\right\}\subset B(X)$ is an equivalence relation (i.e. reflexive, symmetric and transitive) on $L\left(X\right).$ Moreover, if $\left\{S_h\right\},\ \left\{T_h\right\}\ $ are two bounded asymptotically equivalent families, then are asymptotically quasinilpotent equivalent.

\noindent 

\noindent Let be the sets

\noindent 
\[C_b\left(\left.(0,1\right],\ B\left(X\right)\right)=\]
\[=\left\{\left.\varphi :\left.(0,1\right]\to B\left(X\right)\right|\varphi \left(h\right)=T_h\ {\rm such\ that}\ \varphi \ is\ {\rm countinous\ and\ bounded}\right\}=\]
\[=\left\{\left.{\left\{T_h\right\}}_{h\in \left.(0,1\right]}\subset B(X)\right|{\left\{T_h\right\}}_{h\in \left.(0,1\right]}{\rm \ }{\rm is\ a\ bounded\ family}{\rm ,\ i.e.\ }{\mathop{sup}_{h\in \left.(0,1\right]} \left\|T_h\right\|\ }<\infty \right\}.\] 
and
\[C_0\left(\left.(0,1\right],\ B\left(X\right)\right)=\left\{\left.\varphi \in C_b\left(\left.(0,1\right],\ B\left(X\right)\right)\right|{\mathop{\lim }_{h\to 0} \left\|\varphi (h)\right\|=0\ }\right\}=\]
\[=\left\{\left.{\left\{T_h\right\}}_{h\in \left.(0,1\right]}\subset B(X)\right|{\mathop{\lim }_{h\to 0} \left\|T_h\right\|\ }=0\right\}.\]

\noindent $C_b\left(\left.(0,1\right],\ B\left(X\right)\right)\ $ is a Banach algebra non-commutative with  norm

\noindent 
\[\left\|\left\{T_h\right\}\right\|={sup}_{h\in \left.(0,1\right]}\left\|T_h\right\|,\]

\noindent and $C_0\left(\left.(0,1\right],\ B\left(X\right)\right)$ is a closed  bilateral ideal of $C_b\left(\left.(0,1\right],\ B\left(X\right)\right)$. Therefore the quotient algebra $C_b\left(\left.(0,1\right],\ B\left(X\right)\right)/C_0\left(\left.(0,1\right],\ B\left(X\right)\right)$, which will be called from now $B_{\infty }$, is also a Banach algebra with quotient norm

\noindent 
\[\left\|\dot{\left\{T_h\right\}}\right\|={inf}_{{\left\{U_h\right\}}_{h\in \left.(0,1\right]}\in C_0\left(\left.(0,1\right],\ B\left(X\right)\right)}\left\|\left\{T_h\right\}+\left\{U_h\right\}\right\|={inf}_{{\left\{S_h\right\}}_{h\in \left.(0,1\right]}\in \dot{\left\{T_h\right\}}}\left\|\left\{S_h\right\}\right\|.\] 
Then 
\[\left\|\dot{\left\{T_h\right\}}\right\|={inf}_{{\left\{S_h\right\}}_{h\in \left.(0,1\right]}\in \dot{\left\{T_h\right\}}}\left\|\left\{S_h\right\}\right\|\le \left\|\left\{S_h\right\}\right\|={sup}_{h\in \left.(0,1\right]}\left\|S_h\right\|,\]

\noindent for any  ${\left\{S_h\right\}}_{h\in \left.(0,1\right]}\in \dot{\left\{T_h\right\}}$. Moreover,

\noindent 
\[\left\|\dot{\left\{T_h\right\}}\right\|={inf}_{{\left\{S_h\right\}}_{h\in \left.(0,1\right]}\in \dot{\left\{T_h\right\}}}\left\|\left\{S_h\right\}\right\|={inf}_{{\left\{S_h\right\}}_{h\in \left.(0,1\right]}\in \dot{\left\{T_h\right\}}}{sup}_{h\in \left.(0,1\right]}\left\|S_h\right\|.\]

\noindent If two bounded families ${\left\{T_h\right\}}_{h\in \left.(0,1\right]},\ {\left\{S_h\right\}}_{h\in \left.(0,1\right]}\subset B(X)$ are asymptotically equivalent, then $\mathop{{\rm lim}}_{h\to 0}\left\|S_h-T_h\right\|=0$, i.e.  ${\left\{T_h-S_h\right\}}_{h\in \left.(0,1\right]}\in C_0\left(\left.(0,1\right],\ B\left(X\right)\right)$. 

\noindent Let ${\left\{T_h\right\}}_{h\in \left.(0,1\right]},\ {\left\{S_h\right\}}_{h\in \left.(0,1\right]}\in C_b\left(\left.(0,1\right],\ B\left(X\right)\right)$ be asymptotically equivalent. Then
\[{\mathop{{\lim \sup}}_{h\to 0} \left\|S_h\right\|\ }={\mathop{{\lim \sup}}_{h\to 0} \left\|T_h\right\|\ }.\] 
Since
\[{\mathop{{\lim \sup}}_{h\to 0} \left\|S_h\right\|\ }\le {\sup}_{h\in \left.(0,1\right]}\left\|S_h\right\|,\] 
results that

\noindent 
\[{\mathop{{\lim \sup}}_{h\to 0} \left\|S_h\right\|\ }={inf}_{{\left\{S_h\right\}}_{h\in \left.(0,1\right]}\in \dot{\left\{T_h\right\}}}{\mathop{{\lim \sup}}_{h\to 0} \left\|S_h\right\|\ }\le\]
\[\le {inf}_{{\left\{S_h\right\}}_{h\in \left.(0,1\right]}\in \dot{\left\{T_h\right\}}}{\sup}_{h\in \left.(0,1\right]}\left\|S_h\right\|=\left\|\dot{\left\{T_h\right\}}\right\|,\]

\noindent for any ${\left\{S_h\right\}}_{h\in \left.(0,1\right]}\in \dot{\left\{T_h\right\}}$.

\noindent In particular
\[{\mathop{{\lim \lim}}_{h\to 0} \left\|T_h\right\|\ }\le \left\|\dot{\left\{T_h\right\}}\right\|\le \left\|\left\{T_h\right\}\right\|={\sup}_{h\in \left.(0,1\right]}\left\|T_h\right\|.\] 

\begin{definition}
\label{2.6}
We say $\dot{\left\{S_h\right\}},\ \dot{\left\{T_h\right\}}\in \ B_{\infty }$ are spectral equivalent if

\[{\mathop{lim}_{{\mathbf n}\to \infty } {\left(\left\|{\left(\dot{\left\{S_h\right\}}-\ \dot{\left\{T_h\right\}}\right)}^{\left[{\mathbf n}\right]}\right\|\right)}^{\frac{{\mathbf 1}}{{\mathbf n}}}\ }{\mathbf =}{\mathop{lim}_{{\mathbf n}\to \infty } {\left(\left\|{\left(\dot{\left\{T_h\right\}}-\ \dot{\left\{S_h\right\}}\right)}^{\left[{\mathbf n}\right]}\right\|\right)}^{\frac{{\mathbf 1}}{{\mathbf n}}}\ }{\mathbf =}0,\]

where ${(\dot{\left\{S_h\right\}}-\dot{\left\{T_h\right\}})}^{\left[n\right]}=\sum^n_{k=0}{{\left(-1\right)}^{n-k}C^k_n{\dot{\left\{S_h\right\}}}^k{\dot{\left\{T_h\right\}}}^{n-k}}$.

\[{(\dot{\left\{S_h\right\}}-\dot{\left\{T_h\right\}})}^{\left[n\right]}=\sum^n_{k=0}{{\left(-1\right)}^{n-k}C^k_n{\dot{\left\{S_h\right\}}}^k{\dot{\left\{T_h\right\}}}^{n-k}}=\dot{\left\{\sum^n_{k=0}{{\left(-1\right)}^{n-k}C^k_n{S_h}^k{T_h}^{n-k}}\right\}}=\dot{\left\{{(S_h-T_h)}^{\left[n\right]}\right\}}.\] 
\end{definition}

Therefore $\dot{\left\{S_h\right\}},\ \dot{\left\{T_h\right\}}\in \ B_{\infty }$ are spectral equivalent if

\[{\mathop{lim}_{{\mathbf n}\to \infty } {\left\|\dot{\left\{{(S_h-T_h)}^{\left[n\right]}\right\}}\right\|}^{\frac{{\mathbf 1}}{{\mathbf n}}}\ }{\mathbf =}{\mathop{lim}_{{\mathbf n}\to \infty } {\left\|\dot{\left\{{(T_h-S_h)}^{\left[n\right]}\right\}}\right\|}^{\frac{{\mathbf 1}}{{\mathbf n}}}\ }{\mathbf =}0.\]

\begin{proposition}
\label{2.7}
If $\dot{\left\{S_h\right\}},\ \dot{\left\{T_h\right\}}\in \ B_{\infty }$ are spectral equivalent, then any $\left\{S_h\right\}\in \dot{\left\{S_h\right\}}$ and $\left\{T_h\right\}\in \dot{\left\{T_h\right\}}$ are asymptotically spectral equivalent.
\end{proposition}

\noindent
\begin{proof}
Let $\left\{S_h\right\}\in \dot{\left\{S_h\right\}}$ and $\left\{T_h\right\}\in \dot{\left\{T_h\right\}}$ be arbitrary. Thus 

\noindent 
\[{\mathop{\lim }_{n\to \infty } {{\mathop{\overline{\lim }}_{h\to 0} \left\|{\left(S_h-T_h\right)}^{\left[n\right]}\right\|\ }}^{\frac{1}{n}}\ }\le {\mathop{\lim }_{{\mathbf n}\to \infty } {\left\|\dot{\left\{{(S_h-T_h)}^{\left[n\right]}\right\}}\right\|}^{\frac{{\mathbf 1}}{{\mathbf n}}}\ }.\] 
\textbf{}

\noindent Since $\dot{\left\{S_h\right\}},\ \dot{\left\{T_h\right\}}\in \ B_{\infty }$ are spectral equivalent, by Definition \ref{2.6} and above relation, it follows that
 
\[{\mathop{\lim }_{n\to \infty } {{\mathop{\overline{\lim }}_{h\to 0} \left\|{\left(S_h-T_h\right)}^{\left[n\right]}\right\|\ }}^{\frac{1}{n}}\ }=0.\]

\noindent Analogously we can prove that ${\mathop{\lim }_{n\to \infty } {{\mathop{\overline{\lim }}_{h\to 0} \left\|{\left(T_h-S_h\right)}^{\left[n\right]}\right\|\ }}^{\frac{1}{n}}\ }=0$.
\end{proof}
\noindent 

\begin{proposition}
\label{2.8}
Let $\left\{T_h\right\},\left\{S_h\right\}\ \subset B(X)$ be two continuous bounded families. Then $\mathop{lim}_{h\to 0}\left\|T_hS_h-S_hT_h\right\|=0$ if and only if $\dot{\left\{S_h\right\}}\dot{\left\{T_h\right\}}=\dot{\left\{T_h\right\}}\dot{\left\{S_h\right\}}$.
\end{proposition}
\noindent 

\begin{proof} 
$\mathop{{\rm lim}}_{h\to 0}\left\|T_hS_h-S_hT_h\right\|=0$ $\Leftrightarrow $ $\dot{\left\{T_hS_h\right\}}=\dot{\left\{S_hT_h\right\}}$ $\Leftrightarrow $ $\dot{\left\{S_h\right\}}\dot{\left\{T_h\right\}}=\dot{\left\{T_h\right\}}\dot{\left\{S_h\right\}}$.
\end{proof}
\noindent

\begin{definition}
\noindent 
\label{d3.1}We call the \textit{resolvent set }of a family of operators $%
\left\{ S_{h}\right\} \ \in C_b\left(\left.(0,1\right],\ B\left(X\right)\right)$ the set\noindent 
\begin{equation*}
r\left( \left\{ S_{h}\right\} \right) =\{\left. \lambda \in {\mathbb{C}}%
\right\vert \exists \left\{ {\mathcal{R}}(\lambda ,S_{h})\right\} \in C_b\left(\left.(0,1\right],\ B\left(X\right)\right),\ {\ \mathop{\lim}_{h\rightarrow 0}\left\Vert \left(
\lambda I-S_{h}\right) {\mathcal{R}}\left( \lambda ,S_{h}\right)
-I\right\Vert \ }=
\end{equation*}%
\begin{equation*}
={\mathop{lim}_{h\rightarrow 0}\left\Vert {\mathcal{R}}\left( \lambda
,S_{h}\right) \left( \lambda I-S_{h}\right) -I\right\Vert \ }=0\}
\end{equation*}

We call the \textit{spectrum }of a family of operators $\left\{
S_{h}\right\} \ \in C_b\left(\left.(0,1\right],\ B\left(X\right)\right) $ the set%

\begin{equation*}
Sp\left( \left\{ S_{h}\right\} \right) \mathrm{=}{\mathbb{C}}\backslash
r\left( \left\{ S_{h}\right\} \right) \text{.}
\end{equation*}
 
\[Sp\left(\left\{S_h\right\}\right){\rm =}{\mathbb C}\backslash r\left(\left\{S_h\right\}\right).\]

\noindent $r\left(\left\{S_h\right\}\right)$ is an open set of $C$. If $\left\{S_h\right\}$ is a bounded family, then $Sp\left(\left\{S_h\right\}\right)$ is a compact set of $C$. 

\noindent 
\end{definition}

\begin{remark}
\noindent i) If $\lambda \in r\left(S_h\right)$\ for any $h\in (\left.0,1\right]$, then $\lambda \in r\left(\left\{S_h\right\}\right).$ Therefore $\bigcap_{h\in (\left.0,1\right]}{r\left(S_h\right)}\subseteq r\left(\left\{S_h\right\}\right)$\textit{$;$}

\noindent ii) If $\lambda \in Sp\left(\left\{S_h\right\}\right)$, then $\left|\lambda \right|\le {\mathop{\lim \sup }_{n\to \infty } {{\mathop{\lim }_{h\to 0} \left\|{S_h}^n\right\|\ }}^{\frac{1}{n}}\ }$;

\noindent iii) If $\left\|S_h\right\|<\left|\lambda \right|$ for any $h\in (\left.0,1\right]$, then $\lambda \in r\left(\left\{S_h\right\}\right);$ 

\noindent iv) If $\left\{S_h\right\}$ is bounded, then $\left\{{\mathcal R}\left(\lambda ,S_h\right)\right\}$ is also bounded, for every $\lambda \in r\left(\left\{S_h\right\}\right);$

\noindent v) If $\left\{S_h\right\}$ is bounded, then$\ {\ \mathop{lim}_{h\to 0} \left\|{\mathcal R}\left(\lambda ,S_h\right)\right\|\ne 0\ }$, for every $\lambda \in r\left(\left\{S_h\right\}\right).$

\noindent 
\end{remark}

\begin{proposition}
\noindent (resolvent equation - asymptotic) Let $\left\{S_h\right\}\ \subset B(X)$ be a bounded family and $\lambda ,\mu \in r(\left\{S_h\right\})$. Then

\[{\mathop{lim}_{h\to 0} \left\|{\mathcal R}\left(\lambda ,S_h\right)-{\mathcal R}\left(\mu ,S_h\right)-\left(\mu -\lambda \right){\mathcal R}\left(\lambda ,S_h\right){\mathcal R}\left(\mu ,S_h\right)\right\|\ }=0.\] 

\end{proposition}

\begin{proposition}
\noindent Let $\left\{S_h\right\}\ \subset B(X)$ be a bounded family. If  $\lambda \in r\left(\left\{S_h\right\}\right)$ and $\left\{ {{\mathcal R}}_i\left(\lambda ,S_h\right)\} \ \in C_b\left(\left.(0,1\right],\ B\left(X\right)\right),\ i=\overline{1,2}\right\} $ such that

\[{\ \mathop{lim}_{h\to 0} \left\|\left(\lambda I-S_h\right){{\mathcal R}}_i\left(\lambda ,S_h\right)-I\right\|\ }={\mathop{lim}_{h\to 0} \left\|{{\mathcal R}}_i\left(\lambda ,S_h\right)\left(\lambda I-S_h\right)-I\right\|\ }=0\] 
 
\noindent for $i=\overline{1,2}$, then 
\[{\mathop{lim}_{h\to 0} \left\|{{\mathcal R}}_1\left(\lambda ,S_h\right)-{{\mathcal R}}_2\left(\lambda ,S_h\right)\right\|\ }=0.\] 
\end{proposition}

\begin{theorem}
\label{3.12} Let $\left\{S_h\right\}\in C_b\left(\left.(0,1\right],\ B\left(X\right)\right)$. Then
 
\[Sp\left(\dot{\left\{S_h\right\}}\right)=Sp\left(\left\{S_h\right\}\right).\] 
\end{theorem}

\begin{theorem}
\noindent If two bounded families $\left\{S_h\right\},\ \left\{T_h\right\}\ \subset B(X)$ are asymptotically equivalent, then
\[Sp\left(\left\{S_h\right\}\right)=Sp\left(\left\{T_h\right\}\right).\] 
\end{theorem}

\section{Local Spectrum of a Family of Operators}

\noindent 

\noindent Let ${\mathcal O}$ be the set of analytic functions families ${\left\{f_h\right\}}_{h\in \left.(0,1\right]}$ defined on an open complex set with values in a Banach space \textit{X}, having property \[\overline{\mathop{{\rm lim}}_{h\to 0}}\left\|f_h(\lambda )\right\|<\infty, \] for any $\lambda $ from definition set.

\noindent 
\begin {definition}
\noindent A bounded continue family of operators $\left\{T_h\right\}\subset B(X)$ we said to have \textit{single-valued extension property}, if for any family of analytic functions ${\left\{f_h\right\}}_{h\in \left.(0,1\right]}\in {\mathcal O}$, $f_h:D\to X$, where $D\subset {\mathbb C}$ open, with property \[\mathop{{\rm lim}}_{h\to 0}\left\|\left(\lambda I-T_h\right)f_h(\lambda )\right\|\equiv 0,\] it results $\mathop{{\rm lim}}_{h\to 0}\left\|f_h(\lambda )\right\|\equiv 0.$
\noindent 
\end {definition}

\begin{remark}
\noindent Let $\left\{S_h\right\},\left\{T_h\right\}\subset B(X)$  be two bounded continue families of operators asymptotically equivalent\textit{. }If $\left\{S_h\right\}$ has single-valued extension property, then $\left\{T_h\right\}$ has also single-valued extension property\textit{.}

\noindent 
\end{remark}

\begin{proof}
\noindent Let ${\left\{f_h\right\}}_{h\in \left.(0,1\right]}\in {\mathcal O}$ be a family of functions, $f_h:D\to X$, where $D\subset {\mathbb C}$ open, with $\mathop{{\rm lim}}_{h\to 0}\left\|\left(\lambda I-T_h\right)f_h(\lambda )\right\|\equiv 0$. Then

\noindent 
\[\mathop{\overline{{\rm lim}}}_{h\to 0}\left\|\left(\lambda I-S_h\right)f_h(\lambda )\right\|=\mathop{\overline{{\rm lim}}}_{h\to 0}\left\|\left(\lambda I-S_h-T_h+T_h\right)f_h(\lambda )\right\|\le\]
\[ \mathop{{\rm lim}}_{h\to 0}\left\|\left(\lambda I-T_h\right)f_h(\lambda )\right\|+\mathop{\overline{{\rm lim}}}_{h\to 0}\left\|\left(S_h-T_h\right)f_h(\lambda )\right\|\le \mathop{{\rm lim}}_{h\to 0}\left\|\left(S_h-T_h\right)\right\|\overline{\mathop{{\rm lim}}_{h\to 0}}\left\|f_h(\lambda )\right\|,\]

\noindent for any $\lambda \in D$.

\noindent Raking into account $\left\{S_h\right\},\left\{T_h\right\}$ are asymptotically equivalent, it follows

\noindent 
\[\mathop{{\rm lim}}_{h\to 0}\left\|\left(\lambda I-T_h\right)f_h(\lambda )\right\|\equiv 0.\]

\noindent Since $\left\{T_h\right\}$ has single-valued extension property, we obtain  $\mathop{{\rm lim}}_{h\to 0}\left\|f_h(\lambda )\right\|\equiv 0$, thus $\left\{S_h\right\}$ has single-valued extension property.

\noindent 
\end{proof}

\begin{definition}
\label{4.2}
\noindent  Let $\left\{T_h\right\}\ \subset B(X)$ be a family with single-valued extension property and $x\in X$\textit{. }From now we consider\textit{  }$r_{\left\{T_h\right\}}\left(x\right)$ being the set of elements ${\lambda }_0\in {\mathbb C}$ such that there are the analytic functions from $\mathcal{O}$ $\lambda \mapsto x_h\left(\lambda \right)$ defined on an open neighborhood of\textit{  }${\lambda }_0$ $D\subset r_{\left\{T_h\right\}}\left(x\right)$ with values in\textit{ X, } for any $h\in \left.(0,1\right]$, having property

\noindent 
\[\mathop{lim}_{h\to 0}\left\|\left(\lambda I-T_h\right)x_h\left(\lambda \right)-x\right\|\equiv 0. \]

\noindent $r_{\left\{T_h\right\}}\left(x\right)$ is called \textit{the local resolvent set of}\textbf{\textit{  }}$\left\{T_h\right\}$\textit{ at  }$x$\textbf{\textit{.}}

\noindent The \textit{local spectrum of  }$\left\{T_h\right\}$\textit{ at  }$x$\textbf{ }is defined as the set 

\noindent 
\[{Sp}_{\left\{T_h\right\}}\left(x\right){\rm =}{\mathbb C}\backslash r_{\left\{T_h\right\}}\left(x\right).\]

\noindent We also define the \textit{local spectral space}\textbf{ }of $\left\{T_h\right\}$ as

\noindent 
\[X_{\left\{T_h\right\}}\left(a\right)=\left\{x\in X|{Sp}_{\left\{T_h\right\}}\left(x\right)\subset a\right\},\] 
for all sets\textit{  }$a\subset {\mathbb C}$\textit{. }

\noindent 
\end{definition}

\noindent Let be the set 

\noindent 
\[X_b\left(\left.(0,1\right],\ X\right)=\left\{\left.\varphi :\left.(0,1\right]\to X\right|\varphi \left(h\right)=x_h\ {\rm such\ that}\ \varphi \ {\rm is}\ {\rm continue\ and\ bounded}\right\}=\]
\[=\left\{\left.{\left\{x_h\right\}}_{h\in \left.(0,1\right]}\subset X\right|{\left\{x_h\right\}}_{h\in \left.(0,1\right]}{\rm \ a\ bounded\ sequence,\ i.e.\ }{\mathop{sup}_{h\in \left.(0,1\right]} \left\|x_h\right\|\ }<\infty \right\}.\] 
and
\[X_0\left(\left.(0,1\right],\ X\right)=\left\{\left.\varphi \in X_b\left(\left.(0,1\right],\ X\right)\right|{\mathop{\lim }_{h\to 0} \left\|\varphi (h)\right\|=0\ }\right\}=\]
\[=\left\{\left.{\left\{x_h\right\}}_{h\in \left.(0,1\right]}\subset X\right|{\mathop{\lim }_{h\to 0} \left\|x_h\right\|\ }=0\right\}.\]

\noindent $X_b\left(\left.(0,1\right],\ X\right)\ $ is a Banach space in rapport with norm

\noindent 
\[\left\|\varphi \right\|={sup}_{h\in \left.(0,1\right]}\left\|\varphi (h)\right\|\ \Leftrightarrow \ \left\|\left\{x_h\right\}\right\|={sup}_{h\in \left.(0,1\right]}\left\|x_h\right\|,\]

\noindent and $X_0\left(\left.(0,1\right],\ X\right)$ is a closed subspace of  $X_b\left(\left.(0,1\right],\ X\right)$. Therefore, the quotient space $X_b\left(\left.(0,1\right],\ X\right)/X_0\left(\left.(0,1\right],X\right)$, which will be called from now $X_{\infty }$, is a Banach space in rapport  with quotient norm

\noindent 
\[\left\|\dot{\left\{x_h\right\}}\right\|={inf}_{{\left\{u_h\right\}}_{h\in \left.(0,1\right]}\in X_0\left(\left.(0,1\right],\ X\right)}\left\|\left\{x_h\right\}+\left\{u_h\right\}\right\|=\]
\[={inf}_{{\left\{y_h\right\}}_{h\in \left.(0,1\right]}\in \dot{\left\{x\right\}}}\left\|\left\{y_h\right\}\right\|={inf}_{{\left\{y_h\right\}}_{h\in \left.(0,1\right]}\in \dot{\left\{x_h\right\}}}{sup}_{h\in \left.(0,1\right]}\left\|y_h\right\|.\] 
Thus 
\[\left\|\dot{\left\{x_h\right\}}\right\|={inf}_{{\left\{y_h\right\}}_{h\in \left.(0,1\right]}\in \dot{\left\{x_h\right\}}}\left\|\left\{y_h\right\}\right\|\le \left\|\left\{y_h\right\}\right\|={sup}_{h\in \left.(0,1\right]}\left\|y_h\right\|,\]

\noindent for all  ${\left\{y_h\right\}}_{h\in \left.(0,1\right]}\in \dot{\left\{x_h\right\}}$. 

\noindent 

\noindent Let $B_{\infty }=C_b\left(\left.(0,1\right],\ B\left(X\right)\right)/C_0\left(\left.(0,1\right],\ B\left(X\right)\right)$ and we consider the application $\Psi$ defines by

\noindent 
\[\left(\dot{\left\{T_h\right\}},\dot{\left\{x_h\right\}}\right)\longmapsto \dot{\left\{T_hx_h\right\}}:B_{\infty }\times X_{\infty }\to X_{\infty }.\] 

\begin{remark}
\noindent $X_{\infty }$ is a $B_{\infty }-\ {\rm Banach\ module}$ in rapport with the above application.
\end{remark}
\noindent 

\begin{proof}
\noindent Is the application well defined (i.e. not depending by selection of representatives)?

\noindent Let ${\left\{S_h\right\}}_{h\in \left.(0,1\right]}\in \dot{\left\{T_h\right\}}$ and ${\left\{y_h\right\}}_{h\in \left.(0,1\right]}\in \dot{\left\{x_h\right\}}$. Then

\noindent 
\[\mathop{\overline{{\rm lim}}}_{h\to 0}\left\|S_hy_h-T_hx_h\right\|=\mathop{\overline{{\rm lim}}}_{h\to 0}\left\|S_hy_h-T_hy_h+T_hy_h-T_hx_h\right\|\le\]
\[\le \mathop{\overline{{\rm lim}}}_{h\to 0}\left\|S_hy_h-T_hy_h\right\|+\mathop{\overline{{\rm lim}}}_{h\to 0}\left\|T_hy_h-T_hx_h\right\|\le\]
\[\le \mathop{{\rm lim}}_{h\to 0}\left\|S_h-T_h\right\|\mathop{\overline{{\rm lim}}}_{h\to 0}\left\|y_h\right\|+\mathop{\mathop{\overline{{\rm lim}}}_{h\to 0}\left\|T_h\right\|{\rm lim}}_{h\to 0}\left\|y_h-x_h\right\|=0.\]

\noindent Therefore ${\left\{S_hy_h\right\}}_{h\in \left.(0,1\right]}\in \dot{\left\{T_hx_h\right\}}$, for any ${\left\{S_h\right\}}_{h\in \left.(0,1\right]}\in \dot{\left\{T_h\right\}}$ and ${\left\{y_h\right\}}_{h\in \left.(0,1\right]}\in \dot{\left\{x_h\right\}}$.

\noindent Is  $\Psi$ a bilinear application?

\noindent 
\[\Psi \left(\alpha \dot{\left\{T_h\right\}}+\beta \dot{\left\{S_h\right\}},\dot{\left\{x_h\right\}}\right)=\Psi \left(\dot{\left\{\alpha T_h+\beta S_h\right\}},\dot{\left\{x_h\right\}}\right)=\]
\[=\dot{\left\{(\alpha T_h+\beta S_h)x_h\right\}}=\dot{\left\{\alpha T_hx_h+\beta S_hx_h\right\}}=\]
\[=\alpha \dot{\left\{T_hx_h\right\}}+\beta \dot{\left\{S_hx_h\right\}}=\alpha \Psi \left(\dot{\left\{T_h\right\}},\dot{\left\{x_h\right\}}\right)+\beta \Psi \left(\dot{\left\{S_h\right\}},\dot{\left\{x_h\right\}}\right),\]

\noindent for any $\alpha ,\beta \in {\mathbb C}$.

\noindent Analogously we can prove that 

\noindent 
\[\Psi \left(\dot{\left\{T_h\right\}},\alpha \dot{\left\{y_h\right\}}+\beta \dot{\left\{x_h\right\}}\right)=\alpha \Psi \left(\dot{\left\{T_h\right\}},\dot{\left\{y_h\right\}}\right)+\beta \Psi \left(\dot{\left\{T_h\right\}},\dot{\left\{x_h\right\}}\right).\]

\noindent Is $\Psi$ a continue application?

\noindent 
\[\left\|\Psi \left(\dot{\left\{T_h\right\}},\dot{\left\{x_h\right\}}\right)\right\|=\left\|\dot{\left\{T_hx_h\right\}}\right\|=\]
\[={inf}_{\dot{\left\{T_hx_h\right\}}}\left\|\left\{T_hx_h\right\}\right\|={inf}_{\dot{\left\{T_hx_h\right\}}}{sup}_{h\in \left.(0,1\right]}\left\|T_hx_h\right\|\le\]
\[\le {inf}_{\dot{\left\{T_hx_h\right\}}}{sup}_{h\in \left.(0,1\right]}\left\|T_h\right\|\left\|x_h\right\|\le {inf}_{\dot{\left\{T_h\right\}},\dot{\left\{x_h\right\}}}{sup}_{h\in \left.(0,1\right]}\left\|T_h\right\|\left\|x_h\right\|\le\]
\[\le {inf}_{\dot{\left\{T_h\right\}}}{sup}_{h\in \left.(0,1\right]}\left\|T_h\right\|{inf}_{\dot{\left\{x_h\right\}}}{sup}_{h\in \left.(0,1\right]}\left\|x_h\right\|=\left\|\dot{\left\{T_h\right\}}\right\|\left\|\dot{\left\{x_h\right\}}\right\|.\]

\noindent Thus $\left\|\Psi \left(\dot{\left\{T_h\right\}},\dot{\left\{x_h\right\}}\right)\right\|\le \left\|\dot{\left\{T_h\right\}}\right\|\left\|\dot{\left\{x_h\right\}}\right\|$.

\noindent Let $\dot{\left\{T_h\right\}}\in \ B_{\infty }$ be fixed. The application $\dot{\left\{x_h\right\}}\longmapsto \dot{\left\{T_hx_h\right\}}$ is a linear bounded  operator on $X_{\infty }$?

\noindent 
\[\dot{\left\{T_h(\alpha x_h+\beta y_h)\right\}}=\dot{\left\{\alpha T_hx_h+\beta T_hy_h\right\}=}\dot{\alpha \left\{T_hx_h\right\}}+\beta \dot{\left\{T_hy_h\right\}}.\]

\noindent In addition, since
\[\left\|\dot{\left\{T_hx_h\right\}}\right\|\le \left\|\dot{\left\{T_h\right\}}\right\|\left\|\dot{\left\{x_h\right\}}\right\|,\]

\noindent it follows the application $\dot{\left\{x_h\right\}}\longmapsto \dot{\left\{T_hx_h\right\}}$ is a bounded operator.

\noindent Therefore, $B_{\infty }\subseteq B(X_{\infty })$, where $B(X_{\infty })$ is the algebra of linear bounded operators on  $X_{\infty }.$ 
\end{proof}
\noindent

\begin{definition}
\noindent We say that  ${\dot{\left\{T_h\right\}}}_{h\in \left.(0,1\right]}\in B_{\infty }$ has\textit{ single-valued extension property }if for any analytic function $f:D_0\to X_{\infty }$, where $D_0$ is an open complex set with $\left(\lambda \dot{\left\{I\right\}}-\dot{\left\{T_h\right\}}\right)f(\lambda )\equiv 0$, we have $f(\lambda )\equiv 0$, where $0=\dot{\left\{0\right\}}=X_0\left(\left.(0,1\right],\ X\right)$\textit{.}
\end{definition}
\noindent

\noindent Since $f{\rm (}\lambda )\in X_{\infty }$, it follows there is  $\dot{\left\{x_h(\lambda )\right\}}\in \ X_{\infty }$ such that $f\left(\lambda \right)=\dot{\left\{x_h(\lambda )\right\}}$. Then

\noindent 
\[0\equiv \left(\lambda \dot{\left\{I\right\}}-\dot{\left\{T_h\right\}}\right)f\left(\lambda \right)=\dot{\left\{\lambda I-T_h\right\}}\dot{\left\{x_h(\lambda )\right\}}=\dot{\left\{(\lambda I-T_h)x_h(\lambda )\right\}},\]

\noindent i.e. ${\mathop{\lim }_{h\to 0} \left\|(\lambda I-T_h)x_h(\lambda )\right\|\ }=0$.
\begin{definition}
\noindent We say  ${\dot{\left\{T_h\right\}}}_{h\in \left.(0,1\right]}\in B_{\infty }$ has the \textit{single-valued extension property }if for any analytic function $f:D_0\to X_{\infty }$, where $D_0$ is an open complex set with ${\mathop{lim}_{h\to 0} \left\|(\lambda I-T_h)x_h(\lambda )\right\|\ }\equiv 0$ we have ${\mathop{lim}_{h\to 0} \left\|x_h(\lambda )\right\|\ }\equiv 0$\textit{.}

\noindent 

\noindent The \textit{resolvent set} of an element $\dot{\left\{x_h\right\}}\in \ X_{\infty }$ in rapport with  ${\dot{\left\{T_h\right\}}}_{h\in \left.(0,1\right]}\in B_{\infty }$ is

\noindent 
\[r_{\dot{\left\{T_h\right\}}}\left(\dot{\left\{x_h\right\}}\right)=\left\{\left.{\lambda }_0\in {\mathbb C}\right|\exists \ an\ analytic\ function {\left(\lambda \dot{\left\{I\right\}}-\dot{\left\{T_h\right\}}\right)}\dot{\left\{x_h(\lambda )\right\}}\equiv \dot{\left\{x_h\right\}}\ \right\}=\]
\[=\{\left.{\lambda }_0\in {\mathbb C}\right|\exists \ an\ analytic\ function\ \lambda \mapsto \dot{\left\{x_h(\lambda )\right\}}:V_{{\lambda }_0}\to X_{\infty }, \]
\[ {\mathop{\lim }_{h\to 0} \left\|\left(\lambda I-T_h\right)x_h\left(\lambda \right)-x_h\right\|\ }\equiv 0\ \},\] 
when $V_{{\lambda }_0}$ is an open neighborhood of ${\lambda }_0$.

\noindent Let $\dot{\left\{x\right\}}\in \ X_{\infty }$, where $\dot{\left\{x\right\}}=\left\{\left.\left\{x_h\right\}\in X_b\left(\left.(0,1\right],\ X\right)\right|{\mathop{\lim }_{h\to 0} \left\|x_h-x\right\|\ }=0\right\}$. 

\noindent We will call from now
\[X^0_{\infty }=\left\{\left.\dot{\left\{x\right\}}\in X_{\infty }\right|x\in X\right\}\subset X_{\infty }.\] 
Thus
\[r_{\dot{\left\{T_h\right\}}}\left(\dot{\left\{x\right\}}\right)=\{\left.{\lambda }_0\in {\mathbb C}\right|\exists \ {\rm an\ analytic\ function}\ \lambda \mapsto \dot{\left\{x_h(\lambda )\right\}}:V_{{\lambda }_0}\to X_{\infty }, \]
\[{\mathop{\lim }_{h\to 0} \left\|\left(\lambda I-T_h\right)x_h\left(\lambda \right)-x\right\|\ }\equiv 0\ \}.\] 
\end{definition}

\begin{theorem}
\label{4.5}
\noindent ${\dot{\left\{T_h\right\}}}_{h\in \left.(0,1\right]}\in B_{\infty }$ has the single-valued extension property if and only if there is  $\left\{T_h\right\}\in \dot{\left\{T_h\right\}}$ with single-valued extension property\textit{.}
\end{theorem}
\noindent 

\begin{proof}
\noindent Let ${\left\{f_h\right\}}_{h\in \left.(0,1\right]}\in {\mathcal O}$, $f_h:D\to X$\textit{,} be a family of analytic functions, when $D\subset {\mathbb C}$ open, with $\mathop{{\rm lim}}_{h\to 0}\left\|\left(\lambda I-T_h\right)f_h(\lambda )\right\|\equiv 0$.

\noindent Since ${\left\{f_h\right\}}_{h\in \left.(0,1\right]}\in {\mathcal O}$, it follows that  $\overline{\mathop{{\rm lim}}_{h\to 0}}\left\|f_h(\lambda )\right\|<\infty $, $\forall \lambda \in D$, thus $\left\{f_h(\lambda )\right\}\in X_b\left(\left.(0,1\right],\ X\right)$.

\noindent Let $f:D\to X_{\infty }$ be an application defined by $f\left(\lambda \right)=\dot{\left\{f_h(\lambda )\right\}}$. We prove that $f$ is an analytic function.

\noindent Having in view $\left\{f_h\right\}$ are analytic functions on \textit{D}, for any ${\lambda }_0\in D$, we obtain

\noindent 
\[{\mathop{\lim }_{\lambda \to {\lambda }_0} \frac{f\left(\lambda \right)-f\left({\lambda }_0\right)}{\lambda -{\lambda }_0}\ }={\mathop{\lim }_{\lambda \to {\lambda }_0} \frac{\dot{\left\{f_h(\lambda )\right\}}-\dot{\left\{f_h({\lambda }_0)\right\}}}{\lambda -{\lambda }_0}\ }=\]
\[={\mathop{\lim }_{\lambda \to {\lambda }_0} \dot{\left\{\frac{f_h\left(\lambda \right)-f_h\left({\lambda }_0\right)}{\lambda -{\lambda }_0}\right\}}\ }=\dot{\left\{{\mathop{\lim }_{\lambda \to {\lambda }_0} \frac{f_h\left(\lambda \right)-f_h\left({\lambda }_0\right)}{\lambda -{\lambda }_0}\ }\right\}},\]

\noindent for any $\lambda \in D$. Therefore, $f$ is analytic function on \textit{D}.

\noindent By relation $\mathop{{\rm lim}}_{h\to 0}\left\|\left(\lambda I-T_h\right)f_h(\lambda )\right\|\equiv 0$, i.e. $\left(\lambda \dot{\left\{I\right\}}-\dot{\left\{T_h\right\}}\right)f{\rm (}\lambda )\equiv \dot{\left\{0\right\}}$, since $\dot{\left\{T_h\right\}}$ has the single-valued extension property, it follows that $f{\rm (}ëë)\equiv \dot{\left\{0\right\}}$, i.e. \[ \mathop{{\rm lim}}_{h\to 0}\left\|f_h(\lambda )\right\|\equiv 0.\] Hence $\left\{T_h\right\}$ has the single-valued extension property.

\noindent \textbf{Reciprocal:} Let $\left\{T_h\right\}$ has the single-valued extension property. We prove $\dot{\left\{T_h\right\}}$ has also the single-valued extension property.

\noindent Let $f:D\to X_{\infty }$ be an analytic application defined by $f\left(\lambda \right)=\dot{\left\{x_h(\lambda )\right\}}$ such that

\noindent 
\[\left(\lambda \dot{\left\{I\right\}}-\dot{\left\{T_h\right\}}\right)f{\rm (}\lambda )\equiv \dot{\left\{0\right\}}.\] 
Then $\mathop{{\rm lim}}_{h\to 0}\left\|\left(\lambda I-T_h\right)x_h(\lambda )\right\|\equiv 0$.

\noindent We prove that the applications $\lambda \longmapsto x_h\left(\lambda \right):D\to X$ are analytical, $\forall h\in \left.(0,1\right]$.

\noindent Since $f$ is analytical function, it follows that 

\noindent 
\[f^{'}\left({\lambda }_0\right)={\mathop{\lim }_{\lambda \to {\lambda }_0} \frac{f\left(\lambda \right)-f\left({\lambda }_0\right)}{\lambda -{\lambda }_0}\ }={\mathop{\lim }_{\lambda \to {\lambda }_0} \frac{\dot{\left\{x_h(\lambda )\right\}}-\dot{\left\{x_h({\lambda }_0)\right\}}}{\lambda -{\lambda }_0}\ }={\mathop{\lim }_{\lambda \to {\lambda }_0} \dot{\left\{\frac{x_h\left(\lambda \right)-x_h\left({\lambda }_0\right)}{\lambda -{\lambda }_0}\right\}}\ }.\]

\noindent Therefore, there is $\dot{\left\{{\mathop{\lim }_{\lambda \to {\lambda }_0} \frac{x_h\left(\lambda \right)-x_h\left({\lambda }_0\right)}{\lambda -{\lambda }_0}\ }\right\}}\in X_{\infty }$ and thus there is ${\mathop{\lim }_{\lambda \to {\lambda }_0} \frac{x_h\left(\lambda \right)-x_h\left({\lambda }_0\right)}{\lambda -{\lambda }_0}\ }\in X$, $\forall h\in \left.(0,1\right].$

\noindent Since $\left(\lambda \dot{\left\{I\right\}}-\dot{\left\{T_h\right\}}\right)f\left(\lambda \right)\equiv \dot{\left\{0\right\}}$, i.e. ${\mathop{\lim }_{h\to 0} \left\|(\lambda I-T_h)x_h(\lambda )\right\|\ }\equiv 0$, taking into account $\left\{T_h\right\}$ has the single-valued extension property, we have  ${\mathop{\lim }_{h\to 0} \left\|x_h(\lambda )\right\|\ }\equiv 0$, i.e. $\dot{\left\{x_h(\lambda )\right\}}=\dot{\left\{0\right\}}$. Therefore, $\dot{\left\{T_h\right\}}$ has the single-valued extension property.
\end{proof}
\noindent

\begin{proposition}
\label{4.6}
\noindent Let ${\dot{\left\{T_h\right\}}}_{h\in \left.(0,1\right]}\in B_{\infty }$ with the single-valued extension property. Then

\noindent
\[r_{\left\{T_h\right\}}\left(x\right)=r_{\dot{\left\{T_h\right\}}}\left(\dot{\left\{x\right\}}\right),\] 
for all $x\in X$.
\end{proposition}
\noindent

\begin{proof}
\noindent If ${\dot{\left\{T_h\right\}}}_{h\in \left.(0,1\right]}\in B_{\infty }$ has the single-valued extension property, then $\left\{T_h\right\}\in \dot{\left\{T_h\right\}}$ has the single-valued extension property (Theorem \ref{4.5}).

\noindent Let ${\lambda }_0\in r_{\left\{T_h\right\}}\left(x\right).$ Hence there are the analytic functions from $\mathcal{O}$ $\lambda \mapsto x_h\left(\lambda \right)$ defined on an open neighborhood of ${\lambda }_0$ $D\subset r_{\left\{T_h\right\}}\left(x\right)$ with values in\textit{ X, }$\forall h\in \left.(0,1\right]$, having property

\noindent 
\[\mathop{{\rm lim}}_{h\to 0}\left\|\left(\lambda I-T_h\right)x_h\left(\lambda \right)-x\right\|\equiv 0.\]

\noindent Similar to proof of Theorem \ref{4.5}, we prove that the application $f:D\to X_{\infty }$ defined by  $f\left(\lambda \right)=\dot{\left\{x_h(\lambda )\right\}}$ is analytical. Thus ${\lambda }_0\in r_{\dot{\left\{T_h\right\}}}\left(\dot{\left\{x\right\}}\right)$.

\noindent \textbf{Reciprocal: }Let 
\[{\lambda }_0\in r_{\dot{\left\{T_h\right\}}}\left(\dot{\left\{x\right\}}\right)=\{\left.{\lambda }_0\in {\mathbb C}\right|\exists \ {\rm an\ analytic\ function}\ \lambda \mapsto \dot{\left\{x_h(\lambda )\right\}}:V_{{\lambda }_0}\to X_{\infty },\]
\[{\mathop{\lim }_{h\to 0} \left\|\left(\lambda I-T_h\right)x_h\left(\lambda \right)-x\right\|\ }\equiv 0\ \}.\]

\noindent Analog proof of Theorem \ref{4.5}, we prove that the applications $\lambda \mapsto x_h(\lambda ):V_{{\lambda }_0}\to X$ are analytical, $\forall h\in \left.(0,1\right]$. Thus ${\lambda }_0\in r_{\left\{T_h\right\}}\left(x\right).$
\end{proof}
\noindent

\begin{remark}
\noindent Let $\left\{T_h\right\}\ \subset B(X)$  be a continuous bounded family of operators having the single-valued extension property and $x\in X$. Then

\noindent \textit{i) }$r\left(\left\{T_h\right\}\right)\subset r_{\left\{T_h\right\}}\left(x\right)$.

\noindent \textit{ii)  }$X_{\left\{T_h\right\}}\left(a\right)=X_{\left\{T_h\right\}}\left(Sp\left\{T_h\right\}\bigcap a\right)$, for each $a\subset {\mathbb C}{\rm .}$

\noindent \textit{iii) }Let ${\lambda }_0\in r_{\left\{T_h\right\}}\left(x\right)$ and the families of holomorphic function from $\mathcal{O}$ $\lambda \mapsto x_h\left(\lambda \right)$ and\textit{  }$\lambda \mapsto y_h\left(\lambda \right)$ defined on\textit{ D, }an open neighborhood of  ${\lambda }_0$, with values in\textit{ X, }for all $h\in \left.(0,1\right]$, having properties
\[\mathop{lim}_{h\to 0}\left\|\left(\lambda I-T_h\right)x_h\left(\lambda \right)-x\right\|=0\] 
and
\[\mathop{lim}_{h\to 0}\left\|\left(\lambda I-T_h\right)y_h\left(\lambda \right)-x\right\|=0, \] 
for each $\lambda \in D$\textit{. }Then
\[\mathop{lim}_{h\to 0}\left\|x_h\left(\lambda \right)-y_h\left(\lambda \right)\right\|=0,\] 
for each $\lambda \in D$\textit{.}

\noindent \textit{iv) }If $\left\{T_h\right\},\left\{S_h\right\}\in C_b\left(\left.(0,1\right],\ B\left(X\right)\right)$ are asymptotically equivalent, then 

\noindent
\[r_{\left\{T_h\right\}}\left(x\right)=r_{\left\{S_h\right\}}\left(x\right),\ \forall x\in X.\] 
\end{remark}

\begin{proof}
\noindent i) By Proposition \ref{4.6} we have

\noindent 
\[r_{\dot{\left\{T_h\right\}}}\left(\dot{\left\{x\right\}}\right)=r_{\left\{T_h\right\}}\left(x\right),  \forall x\in X.\]

\noindent Moreover, by Theorem \ref{3.12}, we know that

\noindent 
\[r\left(\dot{\left\{T_h\right\}}\right)=r\left(\left\{T_h\right\}\right).\]

\noindent Combing the above relations, we obtain

\noindent 
\[r\left(\left\{T_h\right\}\right)=r\left(\dot{\left\{T_h\right\}}\right)\subset r_{\dot{\left\{T_h\right\}}}\left(\dot{\left\{x\right\}}\right)=r_{\left\{T_h\right\}}\left(x\right),\ \forall x\in X.\]

\noindent ii) By i) it results 
\[{Sp}_{\left\{T_h\right\}}\left(x\right)\subset Sp\left(\left\{T_h\right\}\right).\]

\noindent Therefore $x\in \ X_{\left\{T_h\right\}}\left(a\right)$ if and only if 

\noindent 
\[{Sp}_{\left\{T_h\right\}}\left(x\right)\subset a\bigcap Sp\left(\left\{T_h\right\}\right),\] 
i.e. $x\in \ X_{\left\{T_h\right\}}\left(a\bigcap Sp\left(\left\{T_h\right\}\right)\right)$.

\noindent iii) By Definition \ref{4.2}., it results that the analytic functions  $\lambda \mapsto x_h\left(\lambda \right)$ are defined on an open neighborhood of ${\lambda }_0$  $D_1\subset r\left(\left\{T_h\right\}\right)$ with values in\textit{ X} and the analytic functions $\lambda \mapsto y_h\left(\lambda \right)$ are defined on an open neighborhood of ${\lambda }_0$  $D_2\subset r\left(\left\{T_h\right\}\right)$ on\textit{ X}. 

\noindent Let $D\subset D_1\bigcap D_2\subset r\left(\left\{T_h\right\}\right)$ be an open neighborhood of  ${\lambda }_0$.

\noindent Since 
\[\mathop{{\rm lim}}_{h\to 0}\left\|\left(\lambda I-T_h\right)x_h\left(\lambda \right)-x\right\|=0\] 
and
\[\mathop{{\rm lim}}_{h\to 0}\left\|\left(\lambda I-T_h\right)y_h\left(\lambda \right)-x\right\|=0, \] 
for each $\lambda \in D$, thus

\noindent 
\[\mathop{{\rm lim}}_{h\to 0}\left\|\left(\lambda I-T_h\right)x_h\left(\lambda \right)-\left(\lambda I-T_h\right)y_h\left(\lambda \right)\right\|=\mathop{{\rm lim}}_{h\to 0}\left\|\left(\lambda I-T_h\right)(x_h\left(\lambda \right)-y_h\left(\lambda \right))\right\|=0,\]

\noindent for each $\lambda \in D$\textit{.}

\noindent Having in view that the families of functions $ëë\mapsto x_h\left(\lambda \right)$ and $\lambda \mapsto y_h\left(\lambda \right)$ are analytical on D, hence the functions $\lambda \mapsto x_h\left(\lambda \right)-y_h\left(\lambda \right)$ are analytical. Since $\left\{T_h\right\}$ has the single-valued extension property, it follows that

\noindent 
\[\mathop{{\rm lim}}_{h\to 0}\left\|x_h\left(\lambda \right)-y_h\left(\lambda \right)\right\|=0,\] 
for all $\lambda \in D$\textit{.}

\noindent iv) Let ${\lambda }_0\in \ r_{\left\{T_h\right\}}\left(x\right).$ Then there is a family of functions $\left\{x_h\right\}$ from $\mathcal{O}$, with the property 

\noindent 
\[\mathop{{\rm lim}}_{h\to 0}\left\|\left(\lambda I-T_h\right)x_h\left(\lambda \right)-x\right\|\equiv 0.\] 
Thus 

\noindent 
\[\mathop{\overline{{\rm lim}}}_{h\to 0}\left\|\left(\lambda I-S_h\right)x_h\left(\lambda \right)-x\right\|=\mathop{\overline{{\rm lim}}}_{h\to 0}\left\|\left(\lambda I-S_h-T_h+T_h\right)x_h\left(\lambda \right)-x\right\|\le\]
\[\le \mathop{{\rm lim}}_{h\to 0}\left\|\left(\lambda I-T_h\right)x_h\left(\lambda \right)-x\right\|+\mathop{\overline{{\rm lim}}}_{h\to 0}\left\|\left(S_h-T_h\right)x_h\left(\lambda \right)\right\|\le\]
\[\le \mathop{{\rm lim}}_{h\to 0}\left\|S_h-T_h\right\|\mathop{\overline{{\rm lim}}}_{h\to 0}\left\|x_h\left(\lambda \right)\right\|.\]

\noindent Since $\left\{T_h\right\},\ \left\{S_h\right\}\ $ are asymptotically equivalent, by above relation it follows that

\noindent 
\[\mathop{{\rm lim}}_{h\to 0}\left\|\left(\lambda I-S_h\right)x_h\left(\lambda \right)-x\right\|\equiv 0.\] 
Therefore ${\lambda }_0\in \ r_{\left\{S_h\right\}}\left(x\right).$
\end{proof}
\noindent 

\begin{proposition}
\noindent Let $\left\{T_h\right\}\ \subset B(X)$  be a continuous bounded family of operators having the single-valued extension property. Then

\noindent i) For any $a\subset b$ we have $X_{\left\{T_h\right\}}\left(a\right)\subset X_{\left\{T_h\right\}}\left(b\right)$;

\noindent ii) $X_{\left\{T_h\right\}}\left(a\right)$ is a linear sub-space of\textit{ X, }$\forall a\subset {\mathbb C}$;

\noindent iii)  $\left\{\left.\dot{\left\{x\right\}}\in X_{\infty }\right|x\in X_{\left\{T_h\right\}}\left(a\right)\right\}=X^0_{\infty }\bigcap X_{\dot{\left\{T_h\right\}}}\left(a\right)$, $\forall a\subset {\mathbb C}$.
\end{proposition}
\noindent

\begin{proof}
\noindent i) Let $a,b\subset {\mathbb C}$ such that $a\subset b$ and $x\in X_{\left\{T_h\right\}}\left(a\right)$. Then ${Sp}_{\left\{T_h\right\}}\left(x\right)\subset a$, and thus ${Sp}_{\left\{T_h\right\}}\left(x\right)\subset b$. Therefore $x\in X_{\left\{T_h\right\}}\left(b\right)$.

\noindent ii) Let $x,y\in X_{\left\{T_h\right\}}\left(a\right)$ and $\alpha ,\beta \in {\mathbb C}$. In addition, for any ${\lambda }_0\in r_{\left\{T_h\right\}}\left(x\right)\bigcap r_{\left\{T_h\right\}}\left(y\right)$ there are the analytic functions families $\left\{x_h\right\}\ $ and $\left\{y_h\right\}$ defined on an open neighborhood \textit{D} of ${\lambda }_0$ such that
\[\mathop{{\rm lim}}_{h\to 0}\left\|\left(\lambda I-T_h\right)x_h\left(\lambda \right)-x\right\|=0\] 
and
\[\mathop{{\rm lim}}_{h\to 0}\left\|\left(\lambda I-T_h\right)y_h\left(\lambda \right)-y\right\|=0,\] 
for each $\lambda \in D$.

\noindent Let $z_h\left(\lambda \right)=\alpha x_h\left(\lambda \right)+\beta y_h\left(\lambda \right)$, for any $\lambda \in D$ and $h\in \left.(0,1\right].$ Since $\left\{x_h\right\}\ $ and $\left\{y_h\right\}$ are analytic functions families on $D$, it follows that $\left\{z_h\right\}$ is also an analytic functions family on $D$ and more 

\noindent 
\[\mathop{{\rm lim}}_{h\to 0}\left\|\left(\lambda I-T_h\right)z_h\left(\lambda \right)-(\alpha x+\beta y)\right\|\le\]
\[\le \left|\alpha \right|\mathop{{\rm lim}}_{h\to 0}\left\|\left(\lambda I-T_h\right)x_h\left(\lambda \right)-x\right\|+\left|\beta \right|\mathop{{\rm lim}}_{h\to 0}\left\|\left(\lambda I-T_h\right)y_h\left(\lambda \right)-y\right\|=0,\] 
for each $\lambda \in D$.

\noindent Therefor ${\lambda }_0\in r_{\left\{T_h\right\}}\left(\alpha x+\beta y\right)$ and

\noindent 
\[r_{\left\{T_h\right\}}\left(x\right)\bigcap r_{\left\{T_h\right\}}\left(y\right)\subset r_{\left\{T_h\right\}}\left(\alpha x+\beta y\right).\] 
Moreover 
\[{Sp}_{\left\{T_h\right\}}\left(\alpha x+\beta y\right)\subset {Sp}_{\left\{T_h\right\}}\left(x\right)\bigcup {Sp}_{\left\{T_h\right\}}\left(y\right).\]

\noindent Since $x,y\in X_{\left\{T_h\right\}}\left(a\right)$, i.e. ${Sp}_{\left\{T_h\right\}}\left(x\right)\subset a$ and ${Sp}_{\left\{T_h\right\}}\left(y\right)\subset a$, by above relation, it follows that
\[{Sp}_{\left\{T_h\right\}}\left(\alpha x+\beta y\right)\subset a,\] 
hence $\alpha x+\beta y\in X_{\left\{T_h\right\}}\left(a\right)$.

\noindent iii) Since by Proposition \ref{4.6} we have ($r_{\left\{T_h\right\}}\left(x\right)=r_{\dot{\left\{T_h\right\}}}\left(\dot{\left\{x\right\}}\right)$), it follows that ${x\in X}_{\left\{T_h\right\}}\left(a\right)$ if and only if\textit{  }${\dot{\left\{x\right\}}\in X}_{\dot{\left\{T_h\right\}}}\left(a\right)$. Hence 

\noindent 
\[\left\{\left.\dot{\left\{x\right\}}\in X_{\infty }\right|x\in {X}_{\left\{T_h\right\}}\left(a\right)\right\}{\rm =}\left\{\left.\dot{\left\{x\right\}}\in X_{\infty }\right|{Sp}_{\left\{T_h\right\}}\left(x\right)\subset a\right\}=\]
\[=\left\{\left.\dot{\left\{x\right\}}\in X_{\infty }\right|{Sp}_{\dot{\left\{T_h\right\}}}\left(\dot{\left\{x\right\}}\right)\subset a\right\}=X^0_{\infty }\bigcap X_{\dot{\left\{T_h\right\}}}\left(a\right).\] 
\end{proof}

\begin{theorem}
\label{4.9}
\noindent Let $\left\{S_h\right\},\ \left\{T_h\right\}\ \subset B(X)$ be two continuous bounded families of operators having the single-valued extension property, such that $\mathop{lim}_{h\to 0}\left\|T_hS_h-S_hT_h\right\|=0$\textit{. }If $\left\{S_h\right\},\ \left\{T_h\right\}$ are asymptotically spectral equivalent, then
\noindent 

\[{Sp}_{\left\{T_h\right\}}\left(x\right)={Sp}_{\left\{S_h\right\}}\left(x\right),\ \forall x\in X.\] 
\end{theorem}

\begin{proof}
\noindent  Since  $\left\{S_h\right\},\ \left\{T_h\right\}$ have the single-valued extension property, by Theorem \ref{4.5} it results that ${\dot{\left\{T_h\right\}}}_{h\in \left.(0,1\right]},\ {\dot{\left\{S_h\right\}}}_{h\in \left.(0,1\right]}\in B_{\infty }$ have the single-valued extension property.

\noindent If $\left\{S_h\right\},\ \left\{T_h\right\}$ are asymptotically spectral equivalent, by Proposition \ref{2.7} have that ${\dot{\left\{T_h\right\}}}_{h\in \left.(0,1\right]},\ {\dot{\left\{S_h\right\}}}_{h\in \left.(0,1\right]}$ are spectral equivalent. Moreover, we obtain that for any ${\dot{\left\{T_h\right\}}}_{h\in \left.(0,1\right]},\ {\dot{\left\{S_h\right\}}}_{h\in \left.(0,1\right]}\in B_{\infty }$ have the single-valued extension property and being spectral equivalent, it follows that 

\noindent 

\[{Sp}_{\dot{\left\{T_h\right\}}}\left(\dot{\left\{x\right\}}\right)={Sp}_{\dot{\left\{S_h\right\}}}\left(\dot{\left\{x\right\}}\right),\] for any $x\in X$.

\noindent 

\noindent Therefore, applying  Proposition \ref{4.6}, we have

\noindent
 
\[{Sp}_{\left\{T_h\right\}}\left(x\right)={Sp}_{\dot{\left\{T_h\right\}}}\left(\dot{\left\{x\right\}}\right)={Sp}_{\dot{\left\{S_h\right\}}}\left(\dot{\left\{x\right\}}\right)={Sp}_{\left\{S_h\right\}}\left(x\right),\forall x\in X.\] 
\end{proof}

\begin{remark}
\noindent Let $\left\{S_h\right\},\ \left\{T_h\right\}\ \subset B(X)$ be two continuous bounded families of operators having the single-valued extension property, such that $\mathop{lim}_{h\to 0}\left\|T_hS_h-S_hT_h\right\|=0$. If $\left\{S_h\right\},\ \left\{T_h\right\}$ are asymptotically spectral equivalent, then
\noindent 

\[X_{\left\{T_h\right\}}\left(a\right)=X_{\left\{S_h\right\}}\left(a\right),\] 
for any $a\subset {\mathbb C}$.
\end{remark}
\noindent 

\begin{proof}
\noindent Since $\left\{S_h\right\},\ \left\{T_h\right\}$ are asymptotically spectral equivalent, by Theorem \ref{4.9}, it follows that ${Sp}_{\left\{T_h\right\}}\left(x\right)={Sp}_{\left\{S_h\right\}}\left(x\right)$, for all $x\in X$. Then, for any  $x\in X_{\left\{T_h\right\}}\left(a\right)$, i.e. ${Sp}_{\left\{T_h\right\}}\left(x\right)\subset a$, it results that $x\in X_{\left\{S_h\right\}}\left(a\right)$, thus
\[X_{\left\{T_h\right\}}\left(a\right)\subseteq X_{\left\{S_h\right\}}\left(a\right).\]

\noindent Analog, we can show that $X_{\left\{S_h\right\}}\left(a\right)\subseteq X_{\left\{T_h\right\}}\left(a\right)$.
\end{proof}
\noindent 

\noindent 

\noindent 

\noindent Bibliography 

\noindent 

\begin{enumerate}
\item  David Albrecht, Xuan Duong and Alan McIntosh -- \textit{Operator Theory and Harmonic Analysis}, Lecture presented at the Workshop in Analysis and geometry, A.N.U, Canberra, Jan. -- Feb. 1995. 

\item  I. Colojoarã, \textit{Elemente de teorie spectralã}, Editura Academiei, 1968.

\item   I. Colojoarã and C. Foias, \textit{Theory of Generalized Spectral Operators,} Gordon and Breanch Science Publisher, 1968.

\item  K. B. Laursen and M. M. Neuman, \textit{An Introduction on Local Spectral Theory,} Calerdon Press, Oxford, 2000.

\item  M. Sabac, \textit{Teorie spectralã elementarã cu exercitii si probleme selectate si prezentate de Daniel Beltitã}, Biblioteca Societãtii de Stiinte Matematice din Romania.

\item  S. Macovei, \textit{Spectrum of a Family of Operators}, arXiv:1207.2327.
\end{enumerate}

\noindent

\end{document}